\documentclass[a4paper,12pt]{article}
\usepackage[T1]{fontenc}
\usepackage{amsmath, mathtools, graphicx, color, amsthm, amssymb,tikz}
\usepackage[font=small,labelfont=bf,width=12.5cm]{caption}
\usepackage{algorithmic, algorithm}
\usepackage{MnSymbol}

\usepackage{hyperref}

\newtheorem{theorem}{Theorem}[section]

\newtheorem{corollary}[theorem]{Corollary}
\newtheorem{proposition}[theorem]{Proposition}
\newtheorem{conjecture}[theorem]{Conjecture}
\newtheorem{question}[theorem]{Question}
\newtheorem{problem}[theorem]{Problem}

\newtheorem{observation}{Observation}

\newcommand{\CFP}{\chi_{\rm pcf}}
\newcommand{\PCF}{\chi_{\rm pcf}}
\newcommand{\CF}{\chi_{\rm cf}}
\newcommand{\SK}{{\rm SK}}

\author
{
    Yair Caro \thanks{University of Haifa-Oranim, Israel. E-mail: \texttt{yacaro@kvgeva.org.il}},
    \quad
	Mirko Petru\v{s}evski \thanks{Department of Mathematics and Informatics, Faculty of Mechanical Engineering - Skopje, Republic of North Macedonia. E-Mail: \texttt{mirko.petrushevski@gmail.com}},
	\quad
	Riste \v{S}krekovski\thanks{FMF, University of Ljubljana \& Faculty of Information Studies, Novo mesto, Slovenia. E-Mail: \texttt{skrekovski@gmail.com}}
}

\begin{document}
\title{Remarks on proper conflict-free colorings of graphs}

\maketitle

\begin{abstract}
 A vertex coloring of a graph is said to be \textit{conflict-free} with respect to neighborhoods if for every non-isolated vertex there is a color appearing exactly once in its (open) neighborhood. As defined in [Fabrici et al., \textit{Proper Conflict-free and Unique-maximum Colorings of Planar Graphs with Respect to Neighborhoods}, arXiv preprint], the minimum number of colors in any such proper coloring of graph $G$ is the PCF chromatic number of $G$, denoted $\CFP(G)$. In this paper, we determine the value of this graph parameter for several basic graph classes including trees, cycles, hypercubes and subdivisions of complete graphs. We also give upper bounds on $\CFP(G)$ in terms of other graph parameters. In particular, we show that $\CFP(G) \leq5\Delta(G)/2$ and characterize equality. Several sufficient conditions for PCF $k$-colorability of graphs are established for $4\le k\le 6$. The paper concludes with few open problems.
\end{abstract}

\medskip


\section{Introduction}

All considered graphs in this paper are simple, finite and undirected. We follow~\cite{BonMur08} for all terminology and notation not defined here. A $k$-(vertex-)coloring of a graph $G$ is an assignment $\varphi: V(G)\to\{1,\ldots,k\}$. A coloring $\varphi$ is said to be \textit{proper} if every color class is an independent subset of the vertex set of $G$. A \textit{hypergraph} $\mathcal{H}=(V(\mathcal{H}),\mathcal{E}(\mathcal{H}))$ is a generalization of a graph, its (hyper-)edges are subsets of $V(\mathcal{H})$ of arbitrary positive size. There are various notions of (vertex-)coloring of hypergraphs, which when restricted to graphs coincide with proper graph coloring. One such notion was introduced by Even at al.~\cite{EveLotRonSmo03} (in a geometric setting) in connection with frequency assignment problems for cellular networks, as follows. A coloring of a hypergraph $\mathcal{H}$ is \textit{conflict-free} \textit{(CF)} if for every edge $e\in \mathcal{E}(\mathcal{H})$ there is a color $c$ that occurs exactly once on the vertices of $e$. The \textit{CF chromatic number} of $\mathcal{H}$ is the minimum $k$ for which $\mathcal{H}$ admits a CF $k$-coloring. For graphs, Cheilaris~\cite{Che09} studied the \textit{CF coloring with respect to neighborhoods}, that is, the coloring in which for every non-isolated vertex $x$ there is a color that occurs exactly once in the (open) neighborhood $N(x)$, and proved the upper bound $2\sqrt{n}$ for the CF chromatic number of a graph of order $n$. For more on not necessarily proper CF colorings see, e.g., \cite{CheTot11, GleSzaTar14, KosKumLuc12, PacTar09, Smo13}. Quite recently, Fabrici et al.~\cite{FabLuzRinSot22} initiated a study of proper conflict-free colorings with respect to neighborhoods while focusing mainly on planar and outerplanar graphs. In fact, combining the coloring notions of `conflict-free' and `proper' is only natural as the former was initially introduced (in the hypergraph setting) in order to generalize the latter.

The minimum number of colors in any proper conflict-free coloring (with respect to neighborhoods) of a graph $G$ is the \textit{PCF chromatic number} of $G$, denoted $\PCF(G)$. Note that the obvious inequality $\chi(G)\leq\PCF(G)$ becomes an equality for every complete graph. On the other hand, the mentioned inequality may also be strict. In fact, the ratio $\PCF(G)/\chi(G)$ can acquire arbitrarily high value. Indeed, consider a non-empty graph $G$ and let $S(G)$ be \textit{the complete subdivision} of $G$, i.e., the graph obtained from $G$ by subdividing every edge in $E(G)$ exactly once. Then $\PCF(S(G))\geq\chi(G)$ whereas $\chi(S(G))=2$. Consequently, $\PCF(S(G))/\chi(S(G))\geq\chi(G)/2$.

On the other hand, the ratio $\PCF(G)/\chi(G)$ is bounded from above by the \textit{CF chromatic number} of $G$, denoted $\CF(G)$, which is the minimum number of colors in any (not necessarily proper) conflict-free coloring of $G$ (with respect to neighborhoods). Indeed, in order to show that $\PCF(G) \le \CF(G)\, \chi(G)$ simply take $f$ to be a conflict-free coloring that realizes $\CF(G)$ and $g$ be a proper coloring the realizes $\chi(G)$. Then  $h=(f,g)$ is a conflict-free coloring by the first coordinate and a proper coloring by the second coordinate, and it uses $\CF(G)\, \chi(G)$ colors.

Note in passing a fundamental distinction between the chromatic number and the PCF chromatic number. The former graph parameter is monotonic in regard to the `subgraph relation', that is, if $H\subseteq G$ then $\chi(H)\leq\chi(G)$. This nice monotonicity feature does not hold for the PCF chromatic number in general. For example, $C_4$ is a subgraph of the kite $K_4-e$, but nevertheless we have $\PCF(C_4)=4>3=\PCF(K_4-e)$.

\smallskip

We bring this introductory section to an end by mentioning another related and recently introduced coloring concept for graphs, so-called `odd coloring'. The motivation came from another notion of (vertex-)coloring of hypergraphs, which when restricted to graphs coincides with proper graph coloring. Namely, as introduced by Cheilaris et al.~\cite{CheKesPal13}, an \textit{odd coloring} of hypergraph $\mathcal{H}$ is a coloring such that for every edge $e\in \mathcal{E}(\mathcal{H})$ there is a color $c$ with an odd number of vertices of $e$ colored by $c$. Particular features of the same notion notion (under the name \textit{weak-parity coloring}) have been considered by Fabrici and G\"{o}ring~\cite{FabGor16} (in regard to face-hypergraphs of planar graphs) and also by Bunde et al.~\cite{BunMilWesWu07} (in regard to coloring of graphs with respect to paths, i.e., path-hypergraphs). For various edge colorings of graphs with parity condition required at the vertices we refer the reader to~\cite{AtaPetSkr16, Pet15, KanKatVar18, LuzPetSkr15, LuzPetSkr18, Pet18, PetSkr21}.

As defined in~\cite{PetSkr22}, a proper coloring of a graph $G$ is {\em odd} if in the open
neighborhood $N(v)$ of every non-isolated vertex $v$ a color appears an odd number of times.
The minimum number of colors in any odd coloring of $G$ is the \textit{odd chromatic number} of $G$, denoted $\chi_o(G)$. This new graph parameter spurred instant interest among graph theorists (see~\cite{Cra22,CarPetSkr22,PetPor22,CraLafSon22,ChoChoKwoPar22,DujMorOda22}). Clearly $\chi(G)\leq\chi_o(G)\leq\PCF(G)$, where the latter inequality comes from the obvious fact that every proper conflict-free coloring is odd. The following was conjectured in~\cite{CarPetSkr22}.

\begin{conjecture}
    \label{conj.D+1}
If $G$ is a connected graph of maximum degree $\Delta\geq3$, then $\chi_o(G)\leq \Delta+1$.
\end{conjecture}

By the end of this paper we shall be posing a considerably stronger conjecture. The rest of the article is organized as follows. The next section provides characterizations of several basic graph classes in terms of the value of the PCF chromatic number. In Section~3 we discuss PCF colorability of claw-free and chordal graphs. This is followed by a section on upper bounds for $\PCF(G)$ in degree-constrained graphs $G$. In Section~5 we turn to PCF $k$-colorability where $4\le k\le6$. At the end, we briefly convey some of our thoughts for possible further work on the subject of proper conflict-free colorings.

\section{Basic results}
Here we determine the proper conflict-free chromatic number of some basic/simple graph classes and study the behavior of this graph parameter under certain standard graph constructions. Since this is a relatively new graph parameters, we believe that some of these observations will be useful in future study.

\subsection{Trees, cycles and cubes}
First we characterize trees in terms of the PCF chromatic number.

\begin{observation}
    \label{tree}
Let $T$ be a non-trivial tree. Then,
\begin{equation}
    \label{trees}
\PCF(T)=
\begin{cases}
2 & \text{\quad if \,} T=K_2\,;\\
3 & \text{\quad if \,} {\rm otherwise}\,;\\
\end{cases}
\end{equation}
\end{observation}

\begin{proof}
Observe that if $T=K_2$,  then $\CFP(T)=2$. Otherwise, it holds $\Delta(T)\ge 2$ and let us look at a vertex $v$ of degree $\deg(v)=\Delta(G)$. At least three colors must appear in the closed neighborhood $N[v]$ under any conflict-free proper coloring of $T$. Thus, if $T\neq K_2$ then $\PCF(T)\geq3$. To complete the argument, we show by induction on the number of vertices that every tree admits a conflict-free proper $3$-coloring.  Consider a leaf $u$, and take a PCF $3$-coloring of the smaller tree $T-u$. Let $w$ be the neighbor of $u$ in $T$. Then by forbidding at $u$ the color of $w$ and a color with unique appearance in $N_{T-u}(w)$, the coloring extends a PCF $3$-coloring of $T$.
\end{proof}

In regard to cycles, things are quite similar to the case of odd colorings. Namely,
\begin{equation}
    \label{cycles}
\CFP(C_n)= \chi_o(C_n)=
\begin{cases}
3 & \text{\quad if \,} 3 \mid n \,;\\
4 & \text{\quad if \,} 3 \nmid n \text{ and } n\neq5\,;\\
5 & \text{\quad if \,} n=5 \,.
\end{cases}
\end{equation}

Concerning hypercubes we have the following.

\begin{observation}
For $d\ge 2$, it holds $\CFP(Q_d)=4$.
\end{observation}
\begin{proof}
In order to show that four colors always suffice, perceive $Q_d$ as comprised of two copies of $Q_{d-1}$, call them the `left' and `right' copy, and a perfect matching between them.
Color properly by $1$ and $2$ the left copy of $Q_{d-1}$ and color by $3$ and $4$
the right copy of $Q_{d-1}$. Obviously, the obtained 4-coloring is conflict-free and proper.

Let us argue by contradiction that four colors are always required. Suppose there is a PCF $3$-coloring of $Q_d$.
Say vertex $u$ is colored by $1$ and its `unique-neighbor color' is $2$. So apart from a certain neighbor $v$ of $u$ which is of color $2$, all
the other $d-1$ neighbors $x_1, x_2,\ldots , x_{d-1}$ of  $u$ are of color $3$. The vertex $v$ has with each $x_i$ another
common neighbor $z_i$ (distinct from $u$) and in view of the properness all these $z_i$'s must be of color $1$. However, then $v$ has all its neighbors colored by $1$, a contradiction.
\end{proof}

For comparison, the following have been shown in~\cite{CarPetSkr22} in regard to the odd chromatic number of cubes.

\begin{equation}
\chi_o(Q_d)=
\begin{cases}
2 & \text{\quad if \,} d \text{ is odd} \,;\\
4 & \text{\quad if \,} d \text{ is even}\,.
\end{cases}
\end{equation}

\subsection{Complexity}

The \textit{corona} of a graph is obtained by attaching a pendant edge to each vertex.

\begin{observation}
    \label{corona}
Let $H$ be the corona of a connected graph $G$. Then $\PCF(H)\in\{\chi(G),\chi(G)+1\}$, and for every bipartite graph $G$ it holds $\PCF(H)=\chi(G)+1$. On the other hand, if $G$ is not bipartite and
$\Delta(G) \le  2\chi(G)-3$ then $\PCF(H) = \chi(G)$.
\end{observation}
\begin{proof}
Clearly $\PCF(H)\geq\chi(G)$ since the restriction of any PCF coloring of $H$ to $G$ is proper. As for $\PCF(H) \le  \chi(G)  +1$, simply take a proper coloring of $G$ with colors $1,2,\ldots,\chi(G)$ and assign the color $\chi(G)+1$ to each leaf of the corona. Let us show that the latter inequality becomes an equality whenever $G$ is bipartite. If $G$ is empty then $\PCF(H)=2=\chi(G) +1$. Otherwise, $\chi(G)=2$ and for any non-isolated vertex $v$ of $G$ the closed neighborhood $N_H[v]$ uses at least three colors.

Consider a graph $G$ such that $\chi(G)\geq3$ and $\PCF(H) =  \chi(G)  +1$. Take a proper coloring $c$ of $G$ with colors $1,2,\ldots,\chi(G)$. Since $c$ does not extend to a PCF coloring of $H$ with the same color set, there exists a
vertex $v\in V(G)$ whose closed neighborhood $N_G[v]$ contains all $\chi(G)$ colors, and moreover apart from the color $c(v)$ every other color appears at least twice in $N_G(v)$. For otherwise, either some color is missing from $N_G[v]$ in which case simply assign it to the leaf attached to $v$, or there are $\chi(G)-1$ colors in $N_G(v)$ and at least one of those colors appears just once in which case give the leaf of the corona at $v$ another color from the $\chi(G)- 1 \leq 2$ colors appearing in $N_G(v)$. In particular, we conclude that $d_G(v)\geq2(\chi(G)-1)$. Hence if $\Delta(G) \le  2\chi(G)-3$ then $\PCF(H) = \chi(G)$ indeed.
\end{proof}

\noindent\textbf{Remark.} The condition $\Delta(G) \le  2\chi(G)-3$ cannot be replaced by $\Delta(G) \le  2\chi(G)-2$ as soon as $\Delta(G)\geq4$. For example, let $G=K_3\Box K_3$. This is a $4$-regular $3$-colorable graph and under any proper $3$-coloring $c$ of $G$ every vertex $v$ has exactly two neighbors from each color other than $c(v)$. Hence, for its corona $H$ it holds $\CFP(H)=4=\chi(G)+1$.

Now let us discuss the hardness of PCF colorings.
\begin{observation}
Determining the PCF chromatic number is $\mathcal{NP}$-hard.
\end{observation}
\begin{proof}
Let $G$ be a graph of chromatic number at least $3$, and consider its corona $H$. By Observation~\ref{corona}, it holds $\chi(G) \leq \PCF(H) \leq \chi(G) +1$.  So if we can determine the PCF chromatic number in polynomial time,  let alone if we can determine it for Corona graphs, than we can efficiently find
a proper coloring of $G$ that uses at most $\chi(G) +1$ colors (induced by the PCF-coloring of $H$). However, it has been shown by Khanna at al.~\cite{KhaLinSaf00} that already for 3-colorable graphs it is $\mathcal{NP}$-hard to find a proper 4-coloring.
\end{proof}

In~\cite{CarPetSkr22}, we have shown that the odd chromatic number of a graph is bounded by the sum of the total domination number and chromatic number. Here we prove an analogous inequality for the PCF chromatic number.

\begin{observation}
    \label{domination}
 Let $G$ be a connected graph with total domination $\gamma_t(G)$. Then  $$\PCF(G)  \le \gamma_t(G) + \chi(G).$$
\end{observation}
 \begin{proof}
Let $D$ be a total dominating set of $G$ of cardinality $\gamma_t(G)$.
Color each vertex of $D$ with distinct color, and the rest $V\setminus D$ by at most $\chi(G)$ original colors.
Every vertex in $G$ has a  vertex of unique color in $D$ hence this is proper conflict free coloring.
\end{proof}

\noindent\textbf{Remark.} Observe if $G$ has a vertex $v$ of degree  $\deg(v)=|V(G)| - 1$  (i.e., a universal dominating vertex) then $\PCF(G)\le \chi(G) +2$ simply by taking $v$ and a neighbor of $v$ as a total dominating set.

\bigskip
\noindent\textbf{Remark.} The previous remark can be used to give another proof of the $\mathcal{NP}$-hardness of determining the PCF chromatic number. Namely,
let $G$ be an instance for the $k$-coloring problem. Let $H$ be obtained from $G$ by adding a new vertex $v$ adjacent to all vertices of $G$. Clearly
$\chi(H)  =  \chi(G) +1$. So, we have  $\chi(G)  \le  \PCF(H) \le \chi(H) +2 \le \chi(G) +3$.
However the mentioned paper~\cite{KhaLinSaf00} reveals that it is $\mathcal{NP}$-hard to find a $k +\lfloor k/3\rfloor -1$  coloring for $k$-colorable graphs.
Notice that for $k\ge 6$ it holds  $k +3 < k + \lfloor k/3\rfloor  - 1$ and is in the range of $\mathcal{NP}$-hard.

\subsection{Subdivisions}
First we determine the PCF chromatic number of the complete subdivision of a complete graph. Recall that $S(G)$, the complete subdivision of graph $G$, is obtained from $G$ by subdividing every edge in $E(G)$ exactly once. If $G=K_n$ we denote $S(G)$ by $\SK_n$.
\begin{observation} \label{complete subdivision}For $n\ge 3$, it holds $\CFP(\SK_n)=n$.
\end{observation}
\begin{proof} Denote by $v_1,\ldots,v_n$ the vertices of $K_n$, and denote by $u_{i,j}$ the vertex that subdivides the edge $v_iv_j$ of the original graph $K_n$. Observe that:
\begin{itemize}
  \item $\SK_3$ is PCF 3-colorable since $\SK_3=C_6$;

  \item $\SK_4$  is PCF  4-colorable  as follows: let $c(v_j)  = j$  for $j = 1,2,3,4$, and let  $c(u_{1,2}) = 4$,
  $c(u_{1,3}) = 2$, $c(u_{1,4})  = 3$, $c(u_ {2,3})  = 1$, $c(u_{2,4})=1$, $c(u_{3,4})  = 2$;

  \item $\SK_5$ is PCF 5-colorable as follows: let $c(v_j) =j$  for $j = 1,2,3,4,5$, and let   $c(u_{1,2})  = 3$, $c(u_{2,3}) = 4$, $c(u_{3,4}) = 5$, $c(u_{4,5}) = 1$, $c(u_{5,1}) = 2$,  $c(u_{1,3}) = 4$, $c(u_{1,4})  = 5$, $c(u_{2,4}) = 5$,
  $c(u_{2,5}) = 1$, $c(u_{3,5}) = 1$.
\end{itemize}
Let $n\ge6$. Color the original vertices of $K _n$ with $n$ colors $1,\ldots, n$, by assigning the color $i$ to $v_i$. We extend this to a PCF $n$-coloring of $\SK_n$ by considering separately the cases of even and odd $n$.

\bigskip
\noindent\textbf{Case 1:} \textit{$n$ is even}. Let $M=\{v_1v_2,v_ 3 v_4,\ldots,v_{n-1}v_n\}$. Color the subdividing vertices of degree 2 within the perfect matching $M$ such that  $u_{1,2}$ is colored by $n$ and every other $u_{2j-1,2j}$ is colored  by $1$.
Now consider an arbitrary 2-vertex $u_{j,k}$  (obtained by subdividing the edge $v_ jv_ k$)  which is not colored yet.
Its two adjacent vertices $v_ j$ and $v_k$   forbid for $u_{j,k} $ the colors $j$, $k$ due to properness. Forbid also the colors used within $M$, i.e., the colors $1$ and $n$. Since at most $4$ colors are forbidden at the vertex $u_{j,k}$, a fifth color can be used. Notice that this furnishes a PCF  $n$-coloring of $\SK_n$. Namely, properness is clear. As for the conflict-free part, the color $n$ is unique in the neighborhoods of $v_1$ and $v_2$, the color $1$ is unique in the neighborhoods of $v_3,\ldots,v_n$, and each $2$-vertex has distinctly colored neighbors.

\bigskip
\noindent\textbf{Case 2:} \textit{$n$ is odd}. This time let $M=\{v_1v_2,v_ 3 v_4,\ldots,v_{n-2}v_{n-1},v_nv_1\}$. Once again we start by coloring the vertices of degree $2$ within the edges of $M$: $u_{1,2}$ is colored by $n$, every other $u_{2j-1,2j}$ is colored  by $1$, and $u_{1,n}$ is colored by $2$. Repeating the argument from above, we see that at most 5 colors are forbidden at a subdivision vertex not within $M$, and a sixth color can be used.
It is readily seen that this furnishes a PCF $n$-coloring of $\SK_n$.
\end{proof}

Next we give a bound on the PCF chromatic number of $S(G)$ whenever $G$ has a $1$-factor.

\begin{observation}
 Let $S(G)$ be the complete subdivision of a graph $G$ that has a perfect matching.
 Then $\PCF(S(G)) \le \max\{5, \chi(G)\}$. If additionally $G$ is bipartite it holds $\PCF(G)\leq4$, and this is sharp.
 \end{observation}
 \begin{proof}
Color properly the vertices of $G$ in $S(G)$ by using colors $1,2,\ldots,\chi(G)$. Color next each subdividing vertex $w$ of a matching edge $uv$ by always picking a color distinct from  $c(u)$ and $c(v)$.
Every not yet colored vertex  $z$ is a subdivision vertex of an edge $xy$ not belonging to the matching. Forbid at $z$ the two colors $c(x),c(y)$, and also forbid the colors of the subdividing vertices on the two edges in the matching that are incident to $x$ or $y$.
So at most $4$ colors are forbidden at $z$, and we can use a fifth color. The obtained PCF coloring of $S(G)$ uses at most $\max \{5 ,\chi(G)\}$ colors. Hence
$\PCF(S(G)) \le \max \{5 ,\chi(G)\}$.

Assume now $G$ is bipartite with perfect matching $M$. We construct a PCF $4$-coloring of $S(G)$ as follows. Apply to the vertices of $G$ in $S(G)$ a proper coloring with color set $\{1,2\}$. Assign the color $3$ to each subdividing vertex of the matching $M$ and color by $4$ the rest of the subdividing vertices.
This is clearly a proper conflict free coloring of $S(G)$. The bound of $4$ colors is sharp; for example it is realized by every $C_n$ with $3\nmid n$ since then $S(C_n) = C_{2n}$ has PCF chromatic number $4$.
 \end{proof}


\begin{observation}
    \label{forest}
Suppose $H$ is obtained from $G$ by subdividing once every edge of a spanning forest $F$ without isolated vertices. Then
$\PCF(H) \le \chi(G) +2$. Moreover, if $F$ is a perfect matching then  $\PCF(H) \le \chi(G) +1$.
\end{observation}

\begin{proof}
Consider a spanning forest $F$ (including  the possibility of a spanning tree) without isolated vertices that is subdivided (every edge once).  Root the components $F_1,\ldots,F_t$  at roots $v_1,\ldots,v_t$  which are all leaves of the respective trees. Apply a proper coloring on all original vertices of $G$ in $H$ with $\chi(G)$ colors.

Consider the subdividing vertices of degree 2. We shall color them by two new colors, say $x$ and $y$. In any $F_i$, assign the color $x$ to the unique neighbor of the root, then assign the color $y$ to the next level of subdivision vertices, and so on, alternate between $x, y$ up to the end.

Clearly, all subdivision vertices are colored properly and each has its two neighbors of distinct colors.
Every original vertex  $u$ is properly colored as all its neighbors have colors distinct from the color of $u$.  Each  $u$ appears in some forest component which was subdivided. If it is the root of a component then it has a  unique neighbor colored by $x$;   otherwise, every original vertex has a unique color  $x$ or $y$  determined by its father in the subdivided spanning  component.

If the components of the forest $F$ were isolated edges (namely, if $F$ was a matching) than we used only one extra color, and $\chi(G) +1$ colors suffice.
\end{proof}

\noindent\textbf{Remark.} The obtained bound is sharp: e.g.,  if we subdivide a spanning tree of $K_3$ we get $C_5$ with $\PCF(C_5) = 5$ while $\chi(K_3) = 3$.

\begin{observation}
Suppose H is obtained from G by subdivision of a spanning forest without isolates and possible further edges. Then $\PCF(H) \le \max\{5 , \chi(G) +2\}$.
\end{observation}

\begin{proof}
Suppose  we subdivide a spanning forest $F$ without isolates vertices and further subdivide certain edges in $E(G)\setminus E(F)$.
Color the original vertices of $G$ and the subdivided vertices of $F$ as in the proof of Observation~\ref{forest}.  Uncolored remain the subdivision vertices $w$, not belonging to the subdivided forest. Let an arbitrary such $w$ be adjacent to $u$  and $v$, which are  original vertices of $G$ having distinct colors. Forbid at $w$ the colors of $u, v$  and also forbid the colors $x,y$. Thus $4$ colors are forbidden at $z$, which leaves free for $z$  a color from $\max\{ 5 , \chi(G) +2\}$ colors.
 \end{proof}

 We shall be giving next sufficient conditions for PCF $5$-colorability and $4$-colorability, respectively. But before proceeding with that discussion, for the reader's convenience,  we make a brief digression in order to mention some standard terminology that shall be used. A \textit{$j$-vertex}, \textit{$j^+$-vertex}, or \textit{$j^-$-vertex} is a vertex with degree equal to $j$, at least $j$, or at most $j$, respectively. Given positive integer $\ell$, an \textit{$\ell$-thread} in a graph $G$ is a trail of length $\ell+1$ in $G$ whose $\ell$ internal vertices have degree $2$ in the full graph $G$. Note that under this definition, an $\ell$-thread contains two $(\ell-1)$-threads, and the endpoints of a thread may coincide. So, in general, we distinguish between the notions of a `path-thread' and a `cycle-thread'.

 Recall that a complete subdivision of a graph $G$ is obtained from $G$ by subdividing every edge exactly once. More generally, for $k\ge1$ a \textit{$k$-subdivision} of $G$ is obtained by subdividing every edge at least $k$ times. We already know from Observation~\ref{complete subdivision}, that a $1$-subdivision graph can have arbitrarily large PCF chromatic number. To put it other words, given a graph $G$, let $K(G)=\{v\in V(G):\deg(v)\ge3\}$. If the distance $d(u,v)\ge2$ between all vertices from $K(G)$, then still $\PCF(G)$ can acquire arbitrarily large values. Contrarily, the PCF chromatic number of all $2$-subdivisions of graphs is bounded, in view of the following.

\begin{observation}
    \label{22}
Let $G$ be a graph such that every two $3^+$-vertices are at distance at least $3$. Then $\PCF(G)\le5$.
\end{observation}
\begin{proof}
Consider a minimal counter-example $G$.  It is a connected graph of order at least $6$. Take a pair of adjacent $2^-$-vertices, say $v$ and $w$. By the minimality choice of $G$, the graph $G'=G-\{v,w\}$ admits a PCF $5$-coloring $c$. We extend $c$ to $G$ by assigning to $v$ and $w$ distinct colors that are `available' for them in the sense that properness and conflict-freeness of the coloring are preserved. Noting that $v$ and $w$ can have at most one neighbor each in $V(G')$, we forbid at them the colors of those neighbors, denote them $v'$ and $w'$. Moreover, at $v$ we forbid a possible third color which is unique in having single appearance in $N_{G'}(v')$. Thus there are at least two colors that are available for $v$. Similarly, at least two colors are available for $w$. We pick distinct colors that are available for $v$ and $w$, respectively. This gives a PCF $5$-coloring of $G$ as any possible isolated vertex $u$ of $G'$ has $N_G(u)\subseteq \{v,w\}$.
\end{proof}

\noindent\textbf{Remark.} A generalization of Observation~\ref{22} is given by Theorem~\ref{abdegenericity} in the upcoming section.

\bigskip

One naturally wonders if by increasing the minimum distance between $3^+$-vertices, a color can be saved. We ought to be careful here because of the following negative example. Let $G$ consist of two copies of $C_5$ having one common vertex. Then $\PCF(G)=5$, and at the same time (since $|K(G)|=1$) the distance between any two $3^+$-vertices is arbitrarily large. Note in passing that the girth $g(G)=5$. On the other hand, by putting constraint on the girth as well, we have the following positive result for $5$-subdivisions.

\begin{theorem}
    \label{66}
Let $G$ be a graph of girth $g\ge6$ and distance $d(u,v)\ge6$ for every two vertices $u,v\in K(G)$. Then $\PCF(G)\le4$.
\end{theorem}
\begin{proof}
Consider a minimal counter-example $G$. It is a connected graph of minimum degree $\delta(G)=2$. Indeed, for otherwise, there is a $1$-vertex $u$, and by the minimality choice of $G$, the graph $G-u$ admits a PCF $4$-coloring. But any such coloring readily extends to a PCF $4$-coloring of $G$ by forbidding two colors at $u$. Also note that $K(G)\neq\emptyset$, for otherwise $G$ is a cycle $\neq C_5$, and thus it is PCF $4$-colorable (see equation~\eqref{cycles}). We construct a bipartite graph $H[A,B]$ whose $A$-side is comprised of all vertices in $K(G)$, whereas the side $B$ has a vertex for any maximal path-thread in $G$ that has an endpoint in $K(G)$. (Such a path may have exactly one endpoint in  $K(G)$ if it is contained in a cycle-thread having only one vertex in $K(G)$.) Connect by an edge in $H$ every vertex of the $B$-side to each of its endpoints.

\bigskip
\noindent\textbf{Claim 1.} \textit{There exists a matching in $H$ that saturates $A$.} Any vertex $v$ of $A$ has degree at least $2$; i.e., $v$ is a $2^+$-vertex of $H$. Contrarily, any vertex $P$ of $B$ has degree either $1$ or $2$; hence $P$ is a $2^-$-vertex of $H$. In order to apply Hall Theorem~\cite{Hal35} on matchings in bipartite graphs, we need to verify Hall's condition for the $A$-side of $H$. Letting $A^*\subseteq A$, denote by $N(A^*)$ the set of vertices having a neighbor in $A^*$. By double-counting the edge set $[A^*,N(A^*)]$ we obtain the following inequalities regarding its size: $2|A^*|\le e(A^*,N(A^*))\le 2|N(A^*)|$. Consequently, $|A^*|\le|N(A^*)|$, and we are done.

\bigskip
So we may assign to each vertex $v$ of $K(G)$ a maximal path-thread $P(v)$ with endpoint $v$ such that $v'\neq v''$ implies $P(v')\cap P(v'')=\emptyset$. Note that any such $P(v)$ is a path on at least six vertices.
   We shall use the following auxiliary coloring result.

\bigskip
\noindent\textbf{Claim 2.}\textit{If $P_n:v_1v_2\cdots v_{n-1}v_n$ is a path on $n\ge5$ vertices, then it admits a PCF $4$-coloring $c$ such that $c(v_1),c(v_2),c(v_{n-1}),c(v_n)\in\{1,2,3\}$ and $c(v_1),c(v_n)$ are preassigned.} Upon permutation of the colors $1,2,3$, we may assume that $c(v_1)=1$ and $c(v_n)\in\{1,2\}$. Consider first the case of $5\le n\le7$. If $c(v_n)=1$ then take $(c(v_1),c(v_2),\ldots, c(v_{n-1}),c(v_n))$ be equal to: $(1,2,4,3,1)$ (for $n=5$), $(1,2,4,3,2,1)$ (for $n=6$), or $(1,2,3,4,3,2,1)$ (for $n=7$). Contrarily, if $c(v_n)=2$ then set $(c(v_1),c(v_2),\ldots, c(v_{n-1}),c(v_n))$ be equal to: $(1,2,4,3,2)$ (for $n=5$), $(1,2,4,3,1,2)$ (for $n=6$), or $(1,2,3,4,1,3,2)$ (for $n=7$). The case $n\ge8$ is settled inductively. Namely, let $(c(v_1),c(v_5),\ldots,c(v_n))$ be the already constructed coloring of $P_{n-3}$ and set $(c(v_2),c(v_3),c(v_4))=(3,4,1)$.

\bigskip

We are ready to construct the desired PCF $4$-coloring of $G$. Start by coloring every $P(v), v\in K(G)$ as follows. With enumeration $P(v):vv_1v_2\cdots v_{n-1}v_n$, color $v$ by $4$ and color $P(v)-v$ in accordance with Claim~2 such that $v_1$ receives the color $1$ and $v_n$ receives the color $2$. Any other non-colored $2$-vertex $w$ lies in a maximal path-thread $Q$ consisting entirely of $2$-vertices. Observing that $Q$ is of order at least $5$, color it in accordance with Claim~2 so that its endpoints are assigned with the color $2$. This completes a proper $4$-coloring of $G$. Notice that it is conflict-free as every $2$-vertex has differently colored neighbors, and in the neighborhood of every $3^+$-vertex the color $1$ has single appearance.
\end{proof}

Note in passing that one cannot hope in general for PCF $3$-colorability no matter how far apart are $3^+$-vertices, how large the girth is, and how large the maximum degree is. For simply take $t$ copies of some $C_{3k+2}$, all sharing a common vertex $v$ of degree $2t$. To conclude that the obtained graph is not PCF $3$-colorable, we argue by contradiction. By dropping the common vertex $v$, one gets a collection of $t$ disjoint copies of $P_{3k+1}$. Since we are not using a fourth color, the first and the last vertex on each path are colored the same. So in $N(v)$ every color occurs an even number of times, a contradiction.

\section{Claw-free graphs and chordal graphs}

Graph $G$ is said to satisfy the {\em $1/2$-neighborhood condition}, if for every vertex $v$ in $V(G)$, it holds $\chi([N(v)]) \ge \lfloor \deg(v)/2 \rfloor + 1$.

 \begin{proposition}
If $G$ satisfies the $1/2$-neighborhood condition then $\PCF(G) = \chi(G)$.
 \end{proposition}
 \begin{proof}
 Consider a proper coloring of $G$ with $\chi(G)$ colors. Let $v$ be an arbitrary vertex from $V(G)$. Since
 $\chi([N(v)])  \ge \lfloor \deg(v)/2 \rfloor +1$, it follows  that $N(v)$ contains at least $\lfloor \deg(v)/2 \rfloor +1$ colors. Hence it is impossible that every color appearing in $N(v)$ appears there at least twice. In other words, there is a color that appears exactly once, which makes the coloring conflict-free.
 \end{proof}

\begin{proposition}
 Every claw-free  graph $G$ with odd degrees has $\PCF(G) = \chi(G)$.
\end{proposition}
 \begin{proof}
Observe that if $v$ is a vertex in a claw-free graph $G$ then $\chi([N(v)]) \ge \lceil \deg(v)/2\rceil$. Thus, if $v$ is of odd degree $k$ then $\chi([N(v)]) \ge \lfloor \deg(v)/2 \rfloor +1$.  In particular, if all degrees in $G$ are odd than $G$ satisfies the $1/2$-neighborhood condition and $\PCF(G) = \chi(G)$.
 \end{proof}

Let us add to the discussion regarding the class of claw-free  graphs by establishing an upper bound on the proper conflict-free chromatic number in terms of the maximum vertex degree.

\begin{proposition}
    \label{K13}
Let $G$ be a claw-free graph of maximum degree $\Delta$. Then $$\PCF(G) \le  2\Delta +1.$$
\end{proposition}
\begin{proof}
Consider a minimal counter-example $G$. Then $\Delta\ge 3$. Let $u,v$ be a pair of adjacent vertices in $G$, and let $G'=G-\{u,v\}$ be of maximum vertex degree $\Delta'$. Since $G'$ is claw-free and $\Delta'\le\Delta$, it admits a PCF coloring $c$ with color set $\{1,2,\ldots,2\Delta+1\}$, by the minimality choice of $G$. Our objective is to extend $c$ to a PCF coloring of $G$ with the same color set by assigning to $u$ and $v$ distinct colors that are `available' in the sense that properness and conflict-freeness are preserved. Note that this takes care of any possible isolated vertices of $G'$, which are of degree $1$ or $2$ in $G$, as their neighborhoods are subsets of $\{u,v\}$. We shall be looking at the neighborhoods $N'(u)=N_G(u)\backslash\{v\}$ and $N'(v)=N_{G}(v)\backslash\{u\}$, and distinguish between three cases in regard to the existence of uniquely appearing colors.

\bigskip
\noindent
{\bf Case 1:} {\em In each of $N'(u)$ and $N'(v)$ there is a color (not necessarily the same) appearing exactly once.} Forbid at $u$ all colors in $c(N'(u))$ and a color with unique appearance in $N'(v)$. Additionally, for each $w\in N'(u)$ that has a unique color with single appearance in $N_{G'}(w)$, forbid that color at $u$. Therefore, in total we are forbidding at most $2\Delta-1$ colors, which leaves at least two colors that are available for $u$. By symmetry, there are also at least two available colors for $v$. We pick two distinct colors that are available for $u$ and $v$, respectively, and we are done.

\bigskip
\noindent
{\bf Case 2:} {\em In either $N'(u)$ or $N'(v)$, but not both, there is a  color that appears exactly once.} By symmetry, assume this happens in $N'(u)$. Thus, for any color $\alpha\in c(N'(v))$ there are $x,y\in N'(v)$ such that $c(x)=c(y)=\alpha$. The properness of $c$ implies that $x$ and $y$ are non-adjacent. Therefore, since the edge set of $G[\{u,v,x,y\}]$ already contains the edges $uv,vx,vy$, the fact that $G$ is claw-free assures $u$ is adjacent to at least one of the vertices $x,y$. Consequently, $\alpha\in c(N'(u))$. From the arbitrariness of $\alpha$ it follows that $c(N'(v))\subseteq c(N'(u))$. Let us extend $c$ to $G$.  Forbid at $u$ all colors in $c(N'(u))\cup c(N'(v))=c(N'(u))$. Similarly to before, for each $w\in N'(u)$ that has a unique color with single appearance in $N_{G'}(w)$, forbid that color at $u$. So, we are forbidding at most $2\Delta-2$ colors, which leaves at least three colors that are available for $u$. Turning to $v$, forbid all colors in $c(N'(v))$ and a color with unique appearance in $N'(u)$. Also, for each $w\in N'(v)$ that has a unique color with single appearance in $N_{G'}(w)$, forbid that color at $v$. Therefore, we are forbidding at most $(3\Delta-1)/2$ colors, which leaves at least $(\Delta+3)/2$ colors available for $v$; hence at least three colors are available for $v$. Assign to $u$ and $v$ distinct available colors, which completes a PCF coloring of $G$.

\bigskip
\noindent
{\bf Case 3:} {\em In neither $N'(u)$ nor $N'(v)$ a certain color appears exactly once.} Consequently, both $c(N'(u))$ and $c(N'(v))$ are of size at most $\lfloor (\Delta-1)/2\rfloor$. Moreover, by the reasoning applied in the previous case, we deduce that $c(N'(u))=c(N'(v))$. Forbid at $u$ all colors in $c(N'(u))=c(N'(v))$. Additionally, for each $w\in N'(u)$ that has a unique color that appears exactly once in $N_{G'}(w)$, forbid that color at $u$. Therefore, we are forbidding in total at most $(3\Delta-3)/2$ colors, which leaves at least $(\Delta+5)/2$ colors that are available for $u$; hence at least four colors are available for $u$. Analogously, there are also at least four colors which are available for $v$. By picking two distinct colors that are available for $u$ and $v$, respectively, we extend $c$ to a PCF coloring of $G$.

\bigskip
Since all possible cases have been covered, this completes our proof.
\end{proof}

\noindent\textbf{Remark.} Every claw-free graph $G$ has $\chi(G)\geq\Delta/2$. Consequently, in view of Proposition~\ref{K13}, for every claw-free graph $G$ it holds that $\PCF(G) \le  4\chi(G) +1$. This gives us an example of a large family of graphs $G$ where the ratio $\PCF(G)/\chi(G)$ is bounded by a constant.

\bigskip

In the next section, we shall apply the proof method of Proposition~\ref{K13} in order to obtain several general upper bounds for the PCF chromatic number. But first let us establish the same upper bound od $2\Delta+1$ colors for the PCF chromatic number of chordal graphs.

\begin{proposition}
Let $G$  be a chordal graph of maximum degree $\Delta$. Then $$\PCF(G) \le  2\Delta +1.$$
\end{proposition}
\begin{proof}
We shall make use of the fact that every chordal graph $G$ has a perfect elimination order such that whenever we remove a simplicial vertex $v$ according to the perfect elimination order  (with $N[v]$  forming a clique of order $\deg(v) +1$)  $G-v$ is another chordal graph.

Suppose $G$ is a minimal counter-example, thus $\PCF(G)  \geq 2\Delta +2$.  Let $v$ be the simplicial vertex that is first in a perfect elimination order. Then $N(v)$ forms a clique of cardinality $\deg(v) \leq \Delta$.
Delete $v$ to obtain $H =  G - v$  which is chordal of maximum degree $\Delta'\leq \Delta$. Then by minimality choice of $G$, it holds that $\PCF(H)  \leq 2\Delta' +1 \leq 2\Delta +1$.  Observe that all vertices of $N(v)$ receive distinct colors as they from a clique.

Consider now $v$.  Since all vertices in $N(v)$ have distinct colors, $v$ has $\deg(v)$  forbidden colors dictated by $N(v)$ and at most $\deg(v)$ forbidden colors dictated by the unique color for each  member of $N(v)$; altogether at most $2\deg(v) \leq 2\Delta$  colors are forbidden.
Consequently, there is an available color for $v$ (among the possible $2\Delta +1$  colors), and thus $G$ has a proper conflict-free coloring with at most $2\Delta +1$ colors - a contradiction.
\end{proof}

\noindent\textbf{Remark.} The graph $G$ obtained from $C_5$ by adding two diagonals from the same vertex is chordal, and $\Delta(G) = 4  , \chi(G) = 3 ,  \PCF(G) = 4$.  So the equality $\chi(G) = \PCF(G)$ need not hold for chordal graphs.

\section{Sparse graphs}

 Here we turn to general upper bounds for the PCF chromatic number.  Given integers $a,b$ with $2\le a\le b$, a graph $G$ is said to be \textit{$(a,b)$-degenerate} if every induced subgraph $H$ (including $G$) has the following property: there exists a $1^-$-vertex $u$ in $H$ (i.e. $\deg_{H}(u)\le1$) or an $(a^-,b^-)$ edge $vw$, i.e., such that $\deg_H(v)\le a$ and $\deg_H(w)\le b$. Notice that, if present, $(a,b)$-degeneracy is a hereditary property by definition. 


\begin{theorem}
    \label{abdegenericity}
For every $(a,b)$-degenerate graph $G$, it holds that
\begin{equation}
     \label{ab}
                         \CFP(G) \le  \max\{5,\lfloor a/2\rfloor+2b -1\}.
    \end{equation}
\end{theorem}

\begin{proof}
Arguing by contradiction, let $(a,b)$ be a pair of minimum sum $a+b$ for which there exist $(a,b)$-degenerate counter-examples. Among those graphs, consider a minimal counter-example $G$. Thus $G$ is connected and isolate-free. Letting $\ell=\max\{5,\lfloor a/2\rfloor+2b -1\}$, the graph $G$ is not PCF $\ell$-colorable, but any (other) induced subgraph of $G$ admits a PCF $\ell$-coloring.

\bigskip

\noindent\textbf{Claim 1.} \textit{The minimum degree $\delta(G)\ge2$.} If $G$ has a $1$-vertex $u$, take a PCF $\ell$-coloring of $G-u$ and extend to $G$ by forbidding the following (at most two) colors at $u$: the color of its only neighbor and a possible second color that is unique in having a single occurrence in the second neighborhood of $u$. The obtained contradiction settles the claim.

\bigskip

\noindent\textbf{Claim 2.} \textit{There are no adjacent $2$-vertices in $G$.} Suppose  $v$ and $w$ are adjacent $2$-vertices in $G$. Denote $G'=G-\{v,w\}$ and let $v',w'\in V(G')$ be the respective other neighbor of $v$ and $w$ (it is not excluded that $v'=w'$). Since $\delta(G)=2$ and $G\neq C_3$, there are no isolated vertices in $G'$. We extend a PCF $\ell$-coloring of $G'$ to $G$ as follows. Forbid at the vertices $v$ and $w$ the colors of $v',w'$. Moreover, at $v$ (resp. $w$) we forbid a possible third color which is unique in having single appearance in $N_{G'}(v')$ (resp. $N_{G'}(w')$). Since $\ell\ge5$, at least two colors remain available for $v$, and similarly at least two colors are available for $w$. Assign $v$ and $w$ with distinct colors that are available for them, respectively. This gives a PCF $\ell$-coloring of $G$, a contradiction.

\bigskip

\noindent\textbf{Claim 3.} \textit{$a+b\ge5$.} For otherwise, $a=b=2$. However, then $G$ is $(2,2)$-degenerate and of minimum degree $2$, which by definition of $(a,b)$-degeneracy yields a pair of adjacent $2$-vertices, contradicting Claim~2.

\bigskip

 It is implied by Claim~3 that $\ell=\lfloor a/2\rfloor + 2b- 1\ge6$. The $(a,b)$-degeneracy of $G$ provides us with an edge $vw$ such that $\deg(v)\le a$ and $\deg(w)\le b$. We may assume $\deg(v)\le\deg(w)$. Notice that if $b=3$ then $\deg(w)=b$, by Claim~2. Also observe that, as $G\neq C_3$, there are no isolated vertices in $G'=G-\{v,w\}$. By the minimality choice of $G$, there is a PCF coloring $c$ of $G'$ with color set $\{1,2,\ldots, \lfloor a/2\rfloor + 2b- 1\}$. We shall extend $c$ to a PCF coloring of $G$ with the same color set by assigning to $v$ and $w$ distinct available colors. Looking at the neighborhoods $N'(v)=N_G(v)\backslash\{w\}$ and $N'(w)=N_{G}(w)\backslash\{v\}$, we distinguish between four cases in regard to the existence of uniquely appearing colors.

\bigskip
\noindent
{\bf Case 1:} {\em In each of $N'(v)$ and $N'(w)$ there is a color (not necessarily the same) appearing exactly once.} Forbid at $v$ all colors in $c(N'(v))$. Moreover, unless $b=3$, forbid also a color with single appearance in $N'(w)$. Additionally, for each $z\in N'(v)$ that has a unique color with single appearance in $N_{G'}(z)$, forbid that color at $v$. Therefore, in total we are forbidding at most $2a-1$ colors at $v$ if $b>3$, and at most $2a-2$ colors if $b=3$. Thus, still available for $v$ are at least $2(b-a)+\lfloor a/2\rfloor$ colors if $b>3$, and at least $2(b-a)+\lfloor a/2\rfloor+1$ colors if $b=3$. Hence there always are at least two available colors for $v$. Turning to $w$, we forbid at most $2b-1$ colors: namely, all colors in $c(N'(w))$, a color with single appearance in $N'(v)$, and for each $z\in N'(w)$ that has a unique color with single appearance in $N_{G'}(z)$, forbid that color at $w$. This leaves at least $\lfloor a/2\rfloor\ge2$ available colors for $w$. Thus we are able to pick two distinct colors that are available for $v$ and $w$, respectively.

\bigskip
\noindent
{\bf Case 2:} {\em A color has unique appearance  in $N'(v)$, but no such color exists for $N'(w)$.} Hence $|c(N'(w))|\leq\lfloor (b-1)/2\rfloor$. Forbid at $v$ all colors in $c(N'(v))\cup c(N'(w))$. Similarly to before, for each $z\in N'(v)$ that has a unique color with single appearance in $N_{G'}(z)$, forbid that color at $v$. So, we are forbidding at most $2(a-1)+(b-1)/2=2(b-1)+(a-1)/2 - 3(b-a)/2\le 2b+\lfloor a/2\rfloor-2 -3(b-a)/2$ colors; hence at least $3(b-a)/2+1\ge1$ colors are available for $v$. Turning to $w$, forbid all colors in $c(N'(w))$ and a color with single appearance in $N'(v)$. Also, for each $z\in N'(w)$ that has a unique color that occurs exactly once in $N_{G'}(z)$, forbid that color at $w$. Therefore, we are forbidding at most $(b-1)/2+1+(b-1)=b+(b-1)/2$ colors; hence at least $(b-1)/2+\lfloor a/2\rfloor\ge2$ colors are available for $w$. Assign to $v$ and $w$ distinct available colors, and we are done.

\bigskip
\noindent
{\bf Case 3:} {\em A color has unique appearance  in $N'(w)$, but no such color exists for $N'(v)$.} Hence $|c(N'(v))|\leq\lfloor (a-1)/2\rfloor$. Forbid at $v$ all colors in $c(N'(v))$ and a color with unique appearance in $N'(w)$. Similarly to before, for each $z\in N'(v)$ that has a unique color with single appearance in $N_{G'}(z)$, forbid that color at $v$. So, in total we are forbidding at most $(a-1)/2+1+(a-1)\le(\lfloor a/2\rfloor +2b -1)-(2b-a-1)$ colors, which leaves at least $2b-a-1\ge2$ available colors for $v$. Turning to $w$, forbid all colors in $c(N'(v))\cup c(N'(w))$. Moreover, for each $z\in N'(w)$ that has a unique color with exactly one appearance in $N_{G'}(z)$, forbid that color at $w$. Therefore, we are forbidding at most $(a-1)/2+(b-1)+(b-1)\le\lfloor a/2\rfloor +2b -2$ colors; hence there is an available color for $w$. Assign to $v$ and $w$ distinct available colors from $S$, and thus complete the desired PCF coloring of $G$ with color set $\{1,2,\ldots,\lfloor a/2\rfloor +2b - 1\}$.

\bigskip
\noindent
{\bf Case 4:} {\em In neither $N'(v)$ nor $N'(w)$ a certain color appears exactly once.} Consequently,  $|c(N'(v))|\le\lfloor (a-1)/2\rfloor$ and $|c(N'(w))|\le \lfloor (b-1)/2\rfloor$. Forbid at $v$ all colors in $c(N'(v))\cup c(N'(w))$. Additionally, for each $z\in N'(v)$ that has a unique color which appears exactly once in $N_{G'}(z)$, forbid that color at $v$. Therefore, in total we are forbidding at most $\lfloor a/2\rfloor +\lfloor (b-1)/2\rfloor + (a-1)$ colors, which leaves at least $(b-a) + (b+1)/2$ colors that are available for $v$; hence at least two colors are available for $v$. Analogously, we forbid at most $\lfloor a/2\rfloor +\lfloor (b-1)/2\rfloor + (b-1)$ colors at $w$, which leaves at least $(b+1)/2$ colors that are available for $w$; hence at least two colors are available for $w$.  By picking two distinct colors that are available for $v$ and $w$, respectively, we extend $c$ to a PCF coloring of $G$ with the same color set.

\bigskip
Since all possible cases have been covered, this completes our proof.
\end{proof}

There are plenty of $(2,2)$-degenerate graphs that have PCF chromatic number $5$. One such graph family is defined as follows. Let $\mathcal{F}=\{G: G$ is a non-trivial connected graph such that every block of $G$ is isomorphic to $C_5\}$.

\begin{proposition}
    \label{blocks}
If $G\in\mathcal{F}$ then $\PCF(G)=5$.
\end{proposition}
\begin{proof}
Consider a counter-example $G$ with minimum number of blocks $b(G)$. As $G$ is clearly $(2,2)$-degenerate, Theorem~\ref{abdegenericity} implies that $\PCF(G)\le5$. Hence $G$ admits a PCF $4$-coloring. Let $B=$ be an endblock of $G$, and let $v_1,v_2,v_3,v_4,v_5$ be the vertices of the $5$-cycle $B$ met on a circular traversing that starts at the unique cut vertex $v_1$ contained within $B$. Consider a PCF coloring $c$ of $G$ with color set $\{1,2,3,4\}$. Upon permutation of colors, we may assume that $c(v_i)=i$ for $1\le i\le4$. Hence, it must be that $c(v_5)=2$. Thus, the color $2$ does not have single appearance in $N_G(v_1)$. It follows that the restriction of $c$ to $G'=G-\{v_2,v_3,v_4,v_5\}$ is a PCF $4$-coloring. However, $G'\in\mathcal{F}$ and $b(G')=b(G)-1$, contradicting the minimality choice of $G$.
\end{proof}

Note in passing that every connected graph $G$ of maximum degree $\Delta\ge2$ and degenericity $k\ge2$ is $(k,\Delta)$-degenerate. Thus, by Theorem~\ref{abdegenericity}, $\CFP(G) \le  2\Delta +\lfloor k/2\rfloor$, and the inequality is strict as soon as $\Delta\ge3$. The same bound clearly holds if $\Delta=1$ or $k=1$. The following result characterizes the case of equality $\CFP(G) =  2\Delta +\lfloor k/2\rfloor$.

\begin{corollary}
Let $G$ be a connected graph of maximum degree $\Delta\ge1$ and degenericity $k$. Then
   \begin{equation} \label{2D}
                         \CFP(G) \le  2\Delta +\lfloor k/2\rfloor,
    \end{equation}
    with equality if and only if $G=K_2$ or $G=C_5$.
\end{corollary}

\begin{proof}  As already noticed, the inequality~\eqref{2D} follows immediately from the inequality~\eqref{ab}. Assume $G$ is a graph for which~\eqref{2D} turns into an equality.  If $k=1$ then $G$ is a non-trivial tree, hence (by Observation~\ref{trees}) $\CFP(G)\le3$. Consequently, $2\Delta=2\Delta +\lfloor k/2\rfloor\le3$, implying $\Delta=1$. Hence, $G=K_2$. If $k\ge2$ then $G$ is $(k,\Delta)$-degenerate and Theorem~\ref{abdegenericity} gives $\CFP(G) \le  2\Delta +\lfloor k/2\rfloor-1$ unless $k=\Delta=2$. Therefore $G$ is a cycle. Now,  in view of \eqref{cycles}, $\PCF(G) \le 4$ unless $G = C_5$,
in which case $\PCF(C_5)=5$.
\end{proof}

\noindent\textbf{Remark.} Note that a bound involving only the degenericity $k$ is not possible. For example, every $\SK_n$ has $k=2$ whereas $\CFP(\SK_n)=n$ can be arbitrarily large.

\begin{corollary}
    \label{maximumdegree}
Let $G$ be a connected graph of maximum degree $\Delta\ge1$. Then
   \begin{equation} \label{5D2}
                         \CFP(G) \le  \lfloor 5\Delta/2\rfloor,
    \end{equation}
    with equality if and only if $G=K_2$ or $G=C_5$.
\end{corollary}

\noindent\textbf{Remark.} In conclusion to this section we mention that  a slightly more general concept than  $(a,b)$-degeneracy still allows to apply the technique presented in Theorem~\ref{abdegenericity}.
For an integer $h\ge 4$ a graph $G$ is said to be \textit{$h$-edge degenerate}  if every induced subgraph $H$ (including $G$) has the following property:
there exists a $1^-$-vertex $u$ in $H$ (i.e. $\deg_H(u) \le 1$) or there is an edge $vw$  such that $\deg_H(v) +\deg_H(w) \le  h$.  Clearly, if $G$ is $(a,b)$-degenerate then it is $(a+b)$-edge degenerate. Notice that,  $h$-edge degeneracy is hereditary,  by definition.

Now Theorem~\ref{abdegenericity} can be applied to $h$-edge degenerate graphs by observing that $2b +\lfloor a/2\rfloor \le 2(h-2) +1$; consequently, every $h$-edge degenerate graph $G$ has $\PCF(G) \le  \max\{5, 2(h-2)\}$.

%

\section{Application to PCF colorability of planar graphs}

The recent work of Fabrici et al.~\cite{FabLuzRinSot22}, where the notion of `proper conflict-free coloring' was introduced, focuses on planar and outerplanar graphs. Among other things, \cite{FabLuzRinSot22} contains a `proof from the Book' of the fact that every planar graph is PCF $8$-colorable (cf. Theorem~5.3). The authors of that paper also provide an example of a planar graph (having girth $3$) which requires $6$ colors for a PCF coloring, and purport that $6$ colors always suffice (cf. Conjecture~9.1~(a)). For the class of outerplanar graphs, they establish the tight upper bound of $5$ on the PCF chromatic number (cf. Corollary~5.1).

Throughout this paper we have been using standard graph theory terminology. However, we recall here some notions relevant for this section. An upper bound on the \textit{average degree} $\mathrm{ad}(G)=\frac{2|E(G)|}{|V(G)|}$ of graph $G$ forces sparse local configurations. In order to use such structure results in inductive proofs, one also requires the same bound in all subgraphs. Recall that \textit{maximum average degree}, written $\mathrm{mad}(G)$, is $\mathrm{mad}(G)=\max_{H\subseteq G}\mathrm{ad}(H)$. The \textit{girth} $g(G)$ of $G$ is the length of a shortest cycle in $G$.

In this section, by making use of Theorem~\ref{abdegenericity} and other structure results about sparse local configurations, we deduce several sufficient conditions for PCF $k$-colorability ($4\le k\le6$) in terms of the maximum average degree, and also in terms of the girth of planar graphs. 

First we discuss PCF $k$-colorability in terms of the maximum average degree. For item $(iii)$ of the next result, recall from Proposition~\ref{blocks} that every member of the family $\mathcal{F}=\{G: G$ is a non-trivial connected graph such that every block of $G$ is isomorphic to $C_5\}$ requires $5$ colors for a PCF coloring.

\begin{theorem}
    \label{mad}
Let $G$ be a connected graph with maximum average degree $\mathrm{mad}(G)=m$. Then:
\begin{itemize}
\item[$(i)$] $m<\frac{8}{3}$ implies $\PCF(G)\le6$;
\item[$(ii)$] $m<\frac{5}{2}$ implies $\PCF(G)\le5$;
\item[$(iii)$] $m<\frac{24}{11}$ implies  $\PCF(G)\le 4$, unless $G\in \mathcal{F}$.

\end{itemize}
\end{theorem}
\begin{proof}
To prove $(i)$ we use the following special case of Remark~2.1 from~\cite{CraWes16} (see also bottom pp.~19 in the same reference): \textit{Every graph $G$ with $\mathrm{mad}(G)<\frac{8}{3}$ and $\delta(G)\ge2$ has a $2$-vertex adjacent with a $3^-$-vertex.} Hence, every graph with maximum average degree less than $\frac{8}{3}$ is $(2,3)$-degenerate. Now, from Theorem~\ref{abdegenericity} we deduce that $\PCF(G)\le6$.

\bigskip

Turning to $(ii)$, suppose $G$ is a minimal counter-example. Thus $G$ is a connected graph of order at least $6$. Note that the minimum degree $\delta(G)$ equals $2$. Indeed, for otherwise there is a $1^-$-vertex $u$. But then, the minimality choice of $G$ yields a PCF $5$-coloring of $G-u$, which readily extends to a PCF $5$-coloring of $G$ by forbidding at most two colors at $u$. The following structure result has been proved by Cranston et al.~\cite{CraKimYu10}: \textit{If $G$ is a connected graph with at least four vertices, having $\mathrm{ad}(G)<\frac{5}{2}$ and $\delta(G)\ge2$, then $G$ contains adjacent two vertices or a $3$-vertex having three $2$-neighbors one of which has a second $3$-neighbor.}
Hence our minimal counter-example $G$ has adjacent $2$-vertices or a $3$-vertex of the specified kind. We consider the two possibilities separately:

\smallskip
\noindent
{\bf Case 1:} {\em There are two adjacent vertices $v,w$ with $\deg(v)=\deg(w)=2$.} Denote by $v'$ the other neighbor of $v$, and similarly denote by $w'$ the other neighbor of $w$. (It is not excluded that $v'=w'$.) Let $G'=G-\{v,w\}$. Since $\delta(G)=2$ and $G\neq C_3$, the graph $G'$ is isolate-free. By the minimality choice of $G$, take a PCF $5$-coloring $c$ of $G'$. We extend to $G$ as follows. Forbid at $v$ the colors $c(v'),c(w')$ and a color with unique appearance in $N_{G'}(v')$. This leaves at least two colors available for $v$. Analogously, there are at least two colors that are available for $w$. So distinct available colors can be assigned to $v$ and $w$.

\smallskip
\noindent
{\bf Case 2:} {\em There is a $3$-vertex $v$ having three $2$-neighbors $x,y,z$ such that $x$ has another $3$-neighbor.}  Clearly $x$ is not adjacent to $y$ nor to $z$. We may assume that $y$ and $z$ are non-adjacent as well, for otherwise we are back in the previous case. Let $x',y',z'$ be the other neighbors of $x,y,z$, respectively. We may also assume that $\deg(y'),\deg(z')\ge 3$. Set $G'=G-\{v,x,y,z\}$ and notice that there are no isolated vertices in $G'$. Once again, the minimality choice of $G$ provides us with a PCF $5$-coloring $c$ of $G'$. Forbid at $v$ the colors $c(x'),c(y'),c(z')$. Thus at least two colors are available for $v$. Forbid at $x$ the color $c(x')$, thus leaving at least four available colors. Forbid at $y$ the color $c(y')$ and a color with unique appearance in $N_{G'}(y')$. This leaves at least three colors that are available for $y$. Analogously, there are at least three available colors for $z$. Assign an available color to $v$, and forbid this colors at $x,y,z$. This leaves at least $3$ available colors at $x$, and at least two available colors at $y$ and $z$, respectively. So distinct available colors can be assigned to $x,y$ and $z$.

\smallskip

Since we have obtained a contradiction in every possible case, this settles $(ii)$.

\bigskip

Let us discuss $(iii)$. Again we work with a minimal counter-example $G$, and give several claims that clarify the structure of $G$.

\bigskip
\noindent\textbf{Claim 1.} \textit{$G$ is not obtainable from a member of $\mathcal{F}$ by attaching a pendant edge.} Let $v_1,v_2,v_3,v_4,v_5$ be the vertices of a $C_5$ in circular order, and let $v_0$ be a leaf attached to $v_1$. Color $v_0$ and $v_3$ by 1, $v_1$ by $2$, $v_2$ and $v_5$ by $3$, and $v_4$ by $4$. This is a PCF $4$-coloring of the graph obtained from $C_5$ by attaching a pendant edge. Now consider a graph $H\in\mathcal{F}$
of connectivity $1$. We induct on the number of blocks $b(H)\ge2$ that the graph $H+e$, obtained from $H$ by attaching a pendant edge $e$, is PCF $4$-colorable. Let $B$ be an end-block of $H$ such that $e$ is not attached to an internal vertex of $B$. Let $w_1,w_2,w_3,w_4,w_5$ be an enumeration of $V(B)$ on a circular traversing that starts at the cut vertex $w_1$. The graph $H'=H-\{w_2,w_3,w_4,w_5\}$ is also a member of $\mathcal{F}$ and has $b(H')=b(H)-1$ blocks. By the inductive hypothesis, the graph $H'+e$ admits a PCF coloring $c$ with color set ${1,2,3,4}$. Upon permutation of colors, assume that $c(w_1)=1$, and that the color $2$ appears exactly once in $N_{H'+e}(w_1)$. Extend $c$ to $H+e$ by setting $c(w_2)=3,c(w_3)=4,c(w_4)=2$ and $c(w_5)=3$. This clearly gives a PCF $4$-coloring of $G$, which completes the inductive step.

\bigskip
\noindent\textbf{Claim 2.} \textit{$\delta(G)=2$.} For otherwise, there is a $1$-vertex $u$. However, by Claim~1, the connected graph $G-u$ is not in $\mathcal{F}$. Hence, it is PCF $4$-colorable. Notice that any PCF $4$-coloring of $G-u$ readily extends to $G$ by forbidding at most two colors at $u$.

\bigskip
\noindent\textbf{Claim 3.} \textit{There is a $4$-thread in $G$.} This is implied by Claim~2 and the following well-known result (see e.g. Lemma~2.5 in~\cite{CraWes16}): \textit{Any graph $G$ with $\mathrm{ad}(G)<2+\frac{2}{3\ell-1}$ that has no $2$-regular component contains a $1^-$-vertex or an $\ell$-thread.} Indeed, simply take $\ell=4$ and observe that $G$ is not a cycle.

So let $v_0v_1v_2v_3v_4v_5$ be a $4$-thread in $G$ (the vertices $v_0$ and $v_5$ may coincide).

\bigskip
\noindent\textbf{Claim 4.} \textit{No component of $G'=G-\{v_1,v_2,v_3,v_4\}$ belongs in $\mathcal{F}$.} Suppose the opposite. Then $v_0\neq v_5$, for otherwise $G\in\mathcal{F}$. Notice that $G'$ has at most two components. Let us look into the possibility of two components, say $H$ and $K$,  such that $v_0\in V(H)$ an $v_5\in V(K)$. By Claim~1 and the minimality choice of $G$, both $H+v_0v_1$ and $K+v_4v_5$ admit PCF colorings with color set $\{1,2,3,4\}$. Upon permutation of colors, we may assume that $v_1$ and $v_4$ are colored by $1$, whereas $v_0$ and $v_5$ are colored by $2$. Then simply color $v_3$ by $3$ and $v_4$ by $4$, and we have a PCF $4$-coloring of $G$. This conclusion discards the possibility of two components in $G'$. Hence, we are left with the case of a single component, say $H\in\mathcal{F}$.

Suppose first that at least one of the vertices $v_0,v_5$ is not a cut vertex of $G'$. Say $v_5$ is such. By Claim~1, the graph $H+v_0v_1$ admits a PCF coloring $c$ with color set $\{1,2,3,4\}$. Without loss of generality, assume that $c(v_1)=1$ and $c(v_0)=2$. Hence the list $L(v_2)=\{3,4\}$ consists of the colors that are available for $v_2$. Similarly, $L(v_3)=\{2,3,4\}\backslash \{c(v_5)\}$ and $L(v_4)=\{1,2,3,4\}\backslash\{c(v_5)\}$ (as $\deg{G'}(v_5)=2$). So we are able to color $v_2, v_3, v_4$ with distinct colors from their lists. Suppose now that both $v_0,v_5$ are cut vertices of $G'$. Let $K$ be a component of $G'-v_0$ that does not contain $v_5$. The graph $G-V(K)$ is PCF $4$-colorable, by the minimality choice of $G$. However, a straightforward induction on the number of blocks of the $v$-lobe $[V(K)\cup\{v\}]$ shows that any PCF $4$-coloring of $G-V(K)$ extends to a PCF $4$-coloring of $G$. The obtained contradiction settles the claim.

We are ready to complete our argument.

\bigskip
\noindent\textbf{Claim 5.} \textit{$\PCF(G)\le4$.}  Consider once again the graph $G'=G-\{v_1,v_2,v_3,v_4\}$. By the minimality choice of $G$ and Claim~4, there is a PCF $4$-coloring $c$ of $G'$. Let $v_{-1}$ and $v_6$ be neighbors of $v_0$ and $v_5$ in $G'$, respectively, such that $c(v_0)$ occurs exactly once in $N_{G'}(v_0)$, and similarly $c(v_6)$ occurs exactly once in $N_{G'}(v_5)$. Our objective is to extend $c$ to a PCF $4$-coloring of $G$. The initial lists of available colors for the vertices in $V(G)\backslash V(G')$ are the following: $L(v_1)=\{1,2,3,4\}\backslash\{c(v_{-1}),c(v_0)\}, L(v_2)=\{1,2,3,4\}\backslash\{c(v_0)\}, L(v_3)=\{1,2,3,4\}\backslash\{c(v_5)\},$ and $L(v_4)=\{1,2,3,4\}\backslash\{c(v_{5}),c(v_6)\}$.
Hence $|L(v_1)|=|L(v_4)|=2$, $|L(v_2)|=|L(v_3)|=3$. If $L(v_1)\cap L(v_4)\neq\emptyset$ then assign to both $v_1$ and $v_4$ a common available color $x$, and pick distinct colors from $L(v_2)\backslash\{x\}$ and $L(v_3)\backslash\{x\}$ for $v_2$ and $v_3$, respectively. This clearly gives a PCF $4$-coloring of $G$. Therefore we may assume that $L(v_1)\cap L(v_4)=\emptyset$. Upon permutation of colors, we have $L(v_1)=\{1,2\}, L(v_4)=\{3,4\}$ and $L(v_2)=\{1,2,3\},L(v_3)=\{2,3,4\}$. Set $c(v_i)=i$ for $i=1,\ldots,4$, and we are done.  The obtained contradiction completes our proof.
\end{proof}

Let us turn to girth-constrained planar graphs. For proving the next result we mostly rely on Theorem~\ref{mad} and the following well-known fact: \textit{every planar graph $G$ with girth at least $g$ satisfies $\mathrm{mad}(G)<\frac{2g}{g-2}$} (see~\cite{BorKosNesRasSop99} and also~\cite{CraWes16}, pp.~13).

\begin{theorem}
    \label{planar}
Let $G$ be a planar graph with girth $g(G)=g$. Then:
\begin{itemize}
\item[$(i)$] $g\ge7$ implies $\PCF(G)\le6$;
\item[$(ii)$] $g\ge10$ implies $\PCF(G)\le5$;
\item[$(iii)$] $g\ge24$ implies $\PCF(G)\le4$.
\end{itemize}
\end{theorem}

\begin{proof}
Concerning $(i)$, in view of the above mentioned inequality $\mathrm{mad}(G)<\frac{2g}{g-2}$, planar graphs $G$ of girth $g\ge8$ satisfy $\mathrm{mad}(G)<\frac{8}{3}$; hence they have $\PCF(G)\le 6$ by Theorem~\ref{mad} $(i)$. However, in some cases, planarity permits a stronger result in regard to girth. Namely, Lemma~3.7 in~\cite{CraWes16} reads: \textit{Every planar graph $G$ with girth at least $7$ and minimum degree at least $2$ has a $2$-vertex with a $3^-$-neighbor.} Consequently, every planar graph $G$ of girth $g\ge7$ is $(2,3)$-degenerate, and Theorem~\ref{abdegenericity} applies.

\bigskip

As for $(ii)$, the inequality $\frac{2g}{g-2}\le\frac{5}{2}$ reduces to $g\ge10$. Hence, every planar graph of girth $g\ge10$ has maximum average degree strictly less than $\frac{5}{2}$, and Theorem~\ref{mad} $(ii)$ applies.

\bigskip

And similarly for $(iii)$, the inequality $\frac{2g}{g-2}\le\frac{24}{11}$ reduces to $g\ge24$. Consequently, every planar graph of girth $g\ge24$ has maximum average degree strictly less than $\frac{24}{11}$, and we simply use Theorem~\ref{mad} $(iii)$.
\end{proof}

\bigskip

We end this brief discussion on PCF $4$-colorability with a result about outerplanar graphs.

\begin{theorem}
    \label{out}
Every outerplanar graph $G$ of girth $g\ge6$ is PCF $4$-colorable.
\end{theorem}

\begin{proof}
We invite the reader to check that the proof of Claim~5 used while proving Theorem~\ref{mad} applies verbatim to the following: \textit{No minimal counter-example to Theorem~\ref{out} contains a $4$-thread.} On the other hand, it is easily argued that every outerplanar graph of girth $g$ and minimum degree $\delta\ge2$ contains a $(g-2)$-thread (see e.g.~\cite{MonOchPinRasSop08}, Proposition~19). So it suffices to observe that any minimal counter-example to Theorem~\ref{out} is of minimum degree $2$.
\end{proof}

The girth constrain in Theorem~\ref{out} is sharp in view of Proposition~\ref{blocks}.

\bigskip

Notice that one cannot hope for analogous non-trivial sufficient conditions granting PCF $3$-colorability. More precisely, no value $m>2$ would guarantee that graphs $G$ with $\mathrm{mad}(G)<m$ have $\PCF(G)\le3$. Similarly, there is no value $g$ such that planar or outer planar graphs of girth at least $g$ are PCF $3$-colorable. Namely, every $C_{3k+2}$ requires at least four colors, has $\mathrm{mad}(C_{3k+2})=2$, and is planar of girth $3k+2$.

\section{Further work}
 In view of Observations~\ref{complete subdivision},~\ref{22} and~\ref{66}, we propose the following.

\begin{problem}
For $k=4,5$ determine the minimal pairs $(d,g)$ such that for every graph $G$ having distance at least $d$ between any two $3^+$-vertices and girth at least $g$, it holds that $\PCF(G)\le k$.
\end{problem}

As already observed in the introduction, the ratio $\PCF(G)/\chi(G)$ can acquire arbitrarily high value. Then again, by the remark after Proposition~\ref{K13}, this ratio is bounded by a constant for the class of claw-free graphs $G$. This happens to be the case for the class of planar graphs as well, and for the class of $k$-subdivisions with $k\ge2$.

\begin{problem}
Find other `generic' graph families $\mathcal{G}$ for which there exists a constant $c=c(\mathcal{G})$ such that $\PCF(G)/\chi(G) \le c$ for every $G\in\mathcal{G}$.
\end{problem}

%

Corollary~\ref{maximumdegree} gives an upper bound on the PCF chromatic number in terms of the maximum degree. The established bound (of $\lfloor 5\Delta/2\rfloor$ colors) is sharp for $\Delta\le2$, and open to improvement for $\Delta\ge3$, which gives rise to the following problem.

\begin{problem}
    \label{function}
Find the function $f: \mathbb{N}\to \mathbb{N}$ such that
                          $$f(\Delta) = \max \PCF(G)\,$$
    where the maximum is running over all graphs $G$ of maximum degree $\Delta$.
\end{problem}

Notice that $\Delta+1\le f(\Delta)\le 5\Delta/2$. The three initial values of $f$ are: $f(1)=2,f(2)=5$  and $f(3)=4$, as explained below.

\medskip

It is our belief that for some positive constant $C$, it turns out that $\Delta+C$ colors always suffice. In that direction, the following holds true for connected graphs $G$ having one of the three initial values of $\Delta$. If $\Delta=1$ then $\CFP(G)\leq2$. If $\Delta=2$ then $\CFP(G)\leq5$, and moreover $\CFP(G)\leq4$ if $C_5$ is excepted.
As for $\Delta=3$, recall that a \textit{linear coloring} of a graph is a proper coloring such that each pair of color classes induces a linear forest, that is a union of disjoint paths. A linear coloring  is said to be \textit{superlinear} if the neighbors of every $2$-vertex receive different colors. The following is implied by the main result of~\cite{LiuYu13} (see Theorem~2): \textit{If $G\neq K_{3,3}$ is a connected graph of maximum degree $3$, then $G$ is superlinearly $4$-choosable.} As it is readily checked that $K_{3,3}$ is superlinearly $4$-colorable, it follows that every connected graph of maximum degree $3$ admits a superlinear $4$-coloring. But notice that such a coloring is precisely a PCF coloring of $G$. Thus, if $\Delta=3$ then $\CFP(G) \le 4$.

\medskip

We were not able to find any graph of maximum degree $\Delta=4$ or $5$ that requires more than $\Delta+1$ colors. So we are intrigued and enticed to propose the following extremely bold generalization of Conjecture~\ref{conj.D+1} from the introduction.

\begin{conjecture} If $G$ is a connected graph of maximum degree $\Delta\geq3$, then
                      $$\CFP(G)  \le \Delta+1.$$
\end{conjecture}

Notice that if the above conjecture turns out to be true, then the function $f$ from Problem~\ref{function} satisfies $f(\Delta)=\Delta+1$ for every $\Delta\neq2$.

%
%

\bigskip

We complete the paper by asking a pair of questions in the realm of planar graphs. If answered in the positive, they would provide significant improvements to Theorem~\ref{planar}, items $(ii)$ and $(iii)$, respectively.

\medskip

As already mentioned in Section~5, along with an example of a planar graph having PCF chromatic number $6$, it was conjectured in~\cite{FabLuzRinSot22} that every planar graph is PCF $6$-colorable. We accompany this with the following.

\begin{question}
Does every triangle-free planar graph $G$ have $\PCF(G)\leq5$?
\end{question}

By Proposition~\ref{blocks}, there exist plenty of planar graphs with girth $5$ that require five colors for a PCF coloring. In view of Theorem~\ref{out}, we are tempted to end the article by asking the following.

\begin{question}
\label{PCF4}
Does every planar graph $G$ of girth at least $6$ have $\PCF(G)\leq4$?
\end{question}

Notice that an affirmative answer to Question~\ref{PCF4} would imply the Four Color Theorem (4CT). Indeed, say $G$ is an arbitrary planar graph. Subdivide every edge of $G$ once, i.e., consider the complete subdivision $S(G)$. As girth $g(S(G))\ge6$, take a PCF $4$-coloring $c$ of $S(G)$. The restriction of $c$ to $V(G)$ is a proper $4$-coloring of $G$. 

A weaker form of Question~\ref{PCF4} could be asking whether every planar graph $G$ of girth at least $6$ has $\chi_o(G)\le4$. If answered in the affirmative, even this weaker form would imply 4CT. And a stronger form of Question~\ref{PCF4} (resp. its odd coloring variant) could be asking whether every planar graph $G$ of odd-girth at least $7$ has $\PCF(G)\leq4$ (resp. $\chi_o(G)\leq4$).


\bigskip

\noindent\textbf{Acknowledgments.} We wish to thank prof. J.-S.~Sereni for bringing to our attention the work of Liu and Yu~\cite{LiuYu13} on superlinear colorings of subcubic graphs.
This work is partially supported by ARRS P1-0383, J1-1692, J1-3002.

\bibliographystyle{plain}

\end{document}